\documentclass[11pt]{article}
\usepackage{a4wide,amsmath,amsbsy,amsfonts,amssymb,
stmaryrd,amsthm,mathrsfs,graphicx, amscd, tikz-cd, cancel, supertabular}
\usepackage[normalem]{ulem}
\usetikzlibrary{matrix,arrows,decorations.pathmorphing}
\usepackage{subcaption}

\usepackage{tikz}
\usetikzlibrary{matrix,arrows,decorations.pathmorphing}
\usepackage[all,cmtip]{xy}
\usepackage{qtree}

\usepackage[top=1.2in,bottom=1.2in,left=1.21in,right=1.21in]{geometry}
\usepackage[
	hypertexnames=false,
	hyperindex,
	pagebackref,
	pdftex,
	breaklinks=true,
	bookmarks=false,
	colorlinks,
	linkcolor=blue,
	citecolor=red,
	urlcolor=red,
]{hyperref}
\usepackage{hyperref}

\newtheorem{theorem}{Theorem}[section]
\newtheorem{proposition}[theorem]{Proposition}
\newtheorem{corollary}[theorem]{Corollary}
\newtheorem{lemma}[theorem]{Lemma}
\newtheorem{conjecture}[theorem]{Conjecture}

\theoremstyle{definition}

\newtheorem{definition}[theorem]{Definition}

\newtheorem{question}[theorem]{Question}

\numberwithin{equation}{section}

\theoremstyle{definition}
\newtheorem{remark}[theorem]{Remark}
\newtheorem{remarks}[theorem]{Remarks}

\newcommand{\Cb}{\mathbb{C}}
\newcommand{\CP}{\mathbb{CP}}

\newcommand{\Fb}{\mathbb{F}}

\newcommand{\Pb}{\mathbb{P}}
\newcommand{\Qb}{\mathbb{Q}}
\newcommand{\Q}{\mathbb{Q}}
\newcommand{\Rb}{\mathbb{R}}

\newcommand{\Zb}{\mathbb{Z}}

\newcommand{\beq}{\begin{eqnarray}}
\newcommand{\eeq}{\end{eqnarray}}

\newcommand\ssm{\smallsetminus}

\newcommand{\para}[1]{\medskip\noindent\textbf{#1}}

\DeclareMathOperator{\Kum}{Kum}
\DeclareMathOperator{\RBl}{{\Rb}\!\Bl}
\DeclareMathOperator{\Bl}{B\ell}
\DeclareMathOperator{\Aut}{Aut}

\DeclareMathOperator{\GL}{GL}

\DeclareMathOperator{\Orth}{O}

\DeclareMathOperator{\SU}{SU}
\DeclareMathOperator{\SL}{SL}
\DeclareMathOperator{\SO}{SO}

\DeclareMathOperator{\U}{U}

\DeclareMathOperator{\la}{\langle}
\DeclareMathOperator{\ra}{\rangle}

\DeclareMathOperator{\Diff}{Diff}
\DeclareMathOperator{\Mod}{Mod}

\DeclareMathOperator{\specrad}{specrad}
\DeclareMathOperator{\htop}{h_{\rm top}}

\title{\vspace{-2cm}Entropy-minimizing diffeomorphisms of 
pseudo-Anosov type\\
 on K3 surfaces}
\author{Benson Farb\thanks{Supported in part by National Science Foundation Grant DMS-181772 and the Eckhardt Faculty Fund.} \ and Eduard Looijenga
}

\begin{document}
\maketitle
\begin{abstract}
We construct diffeomorphisms of ``pseudo-Anosov type'' on K3 surfaces $M$.  In particular we obtain infinitely many examples of such $f$ that minimize entropy in their homotopy class, and for which neither $f$ nor 
any diffeomorphism homotopic to $f$ preserves any complex structure on $M$.  
\end{abstract}

\section{Introduction}
A closed, simply-connected complex surface is a {\em K3 surface} if 
it admits a nowhere-vanishing holomorphic $2$-form.  Kodaira proved that all K3 surfaces are diffeomorphic; we call these smooth manifolds {\em K3 manifolds}.  Attached to every diffeomorphism\footnote{In this paper ``diffeomorphism'' will mean ``$C^\infty$ diffeomorphism''.} 
 $f\in\Diff(M)$ is its {\em topological entropy} $\htop(f)\in [0,\infty)$, which records the asymptotic exponential growth rate of the number of $f$-orbits measured up to a given accuracy; see \S\ref{subsection:entropy} for the precise definition.  The main result of this paper is the following.
 
 \begin{theorem}
\label{theorem:main1}
Let $M$ be a K3 manifold.  For each $T\in\SL(4,\Zb)$ there exists a diffeomorphism $f_T:M\to M$ 
(constructed explicitly in \S\ref{subsection:construct} below) with the following properties: 
\begin{enumerate}
\item The topological entropy of $f_T$ is given by
\begin{equation}
\label{eq:ftentropy}
\textstyle \htop(f_T)=\sum_{|\lambda_i|>1}\log |\lambda_i|,
\end{equation}
where $\{\lambda_1,\lambda_2,\lambda_3,\lambda_4\}\subset\Cb$ 
denote the eigenvalues of $T$. 

\item  If $T$ has exactly two 
of its  eigenvalues of absolute value $>1$ then $f_T$ minimizes entropy in its homotopy class.
\item If $T$ has $4$ distinct real eigenvalues then neither $f_T$ nor any 
diffeomorphism homotopic to it preserves any complex structure on $M$. 
\item If $T$ has all eigenvalues off the unit circle then $f_T$ is of pseudo-Anosov type (see \S\ref{subsection:firstprop} below for the precise definition).
\end{enumerate}
\end{theorem}

\begin{remarks}
\ 
\begin{enumerate}
\item Property 2 of Theorem \ref{theorem:main1} was known for the very special case when $f_T$ is a biholomorphic automorphism.  This follows from the fact that complex K3 surfaces are Kahler (Siu \cite{Si}), together with Gromov' Theorem  \cite{Gr} that any biholomorphic automorphism of any closed, Kahler manifold minimizes entropy in its homotopy class.

\item There are infinitely many examples (even up to conjugation and taking powers) of $T\in\SL(4,\Zb)$ satisfying all of 1-4 in Theorem \ref{theorem:main1}: take $T$ to be the companion matrix of 
\[P(x)=(x^2+ax+1)(x^2+bx+1)\ \ \text{for any $a,b\geq 2, a\neq b$}.\]
There are also infinitely many irreducible examples:  
\[P(x)=x^4-(n^2+n)x^2+1\ \ \text{for any $n\geq 2$}.\]  Thus Theorem \ref{theorem:main1} produces infinitely many entropy-minimizing diffeomorphisms $f_T$ of the K3 manifold $M$.    Previously known entropy-minimizing diffeomorphisms on closed manifolds were given in the 1970's and 1980's:  for zero-entropy diffeomorphisms (Shub-Sullivan \cite{SS}); for biholomorphic diffeomorphisms on Kahler manifolds (Gromov \cite{Gr});  Anosov diffeomorphisms on nilmanifolds (Franks, see Remark \ref{remark:franks}); and linear-like diffeomorphisms 
of $S^3\times S^3$ (Fried \cite{Fr}).   Note that pseudo-Anosov homeomorphisms on surfaces minimize entropy in their homotopy class  (Fathi-Shub \cite{FLP}), but these are only smooth off of a finite set.

\item The construction of $f_T$ actually gives uncountably many 
nonconjugate (in $\Diff(M)$) entropy-minimizers in a given isotopy class; see Remark \ref{remark:nonconjugate} below.  We do not know other 
examples of this phenomenon other than (skew) products. 

\end{enumerate}
\end{remarks}

The construction of $f_T$ in Theorem \ref{theorem:main1} starts with 
a real variation of the {\em Kummer construction} of complex K3 surfaces.  
One starts with a complex torus $A:=\Cb^2/\Lambda$, blows up at the set $A[2]$ of 
sixteen $2$-torsion points, and then takes a quotient by the action induced by the involution $(z,w)\mapsto (-z,-w)$, giving a K3 surface $M$.  The group $\SL(4,\Zb)$ acts on $A$ by linear diffeomorphisms.  Any $T\in\SL(4,\Zb)$ preserving the complex structure on $A$ will induce a biholomorphic automorphism of $M$.  The difficulty is that the typical 
$T\in\SL(4,\Zb)$ does not do this, and a more subtle construction must be done.  We are able to do this in a way that is natural in the sense that 
the association $T\mapsto f_T$  induces an 
injective homomorphism 
\[
\rho:\SL(4,\Zb)\to \Mod(M)
\]
where $\Mod(M)=\pi_0(\Diff(M))$  is the mapping class group of $M$; see Remark \ref{remark:naturality} below.  Let $H_M:=H(M;\Zb)\cong\Zb^{22}$ equipped with its intersection form.  Then $H_M$ is the unique even, unimodular lattice of signature $(3,19)$; its 
orthogonal group $\Orth(H_M)$ is an arithmetic subgroup of the real semisimple Lie group $\Orth(H_M\otimes\Rb)\cong\Orth(3,19)(\Rb)$.  The composition 
\[\SL(4,\Zb)\to \Mod(M)\to\Orth(H(M))\]
extends to a representation $\SL(4,\Rb)\to \Orth(3,19)(\Rb)$ of Lie groups.

\para{Homological entropy.}
Theorem \ref{theorem:main1} relates closely to homological entropy.  
For a linear transformation $A$ of a finite-dimensional  real vector space, let 
$\specrad(A)$ denote the {\em spectral radius} of $A$; that is, the maximum of the absolute value of the eigenvalues of $A$.

\begin{proof}[Proof of Property 2 in Theorem \ref{theorem:main1}]
Call the eigenvalues in the hypothesis $\lambda_1$ and $\lambda_2$.  
Item 1 of Theorem \ref{theorem:main1} gives that 
$\htop(f_T)=\log|\lambda_1|+\log|\lambda_2|$.  On the other hand, Shub's Entropy Conjecture (proved by Yomdin for any $C^\infty$ diffeomorphism $g$ of any closed manifold $M$) gives a lower bound 
\[\htop(g)\geq \log\specrad[g_*:H_*(M; \Rb)\to H_*(M; \Rb).]\]
In Corollary \ref{corollary:Yomdin} below we compute this for $(f_T)_*=g_*$ for any $g$ homotopic to $f_T$ to obtain 
\begin{equation}
\label{eq:hbound}
\begin{array}{ll}
\htop(g)&\geq \log\specrad[g_*:H_*(M; \Rb)\to H_*(M; \Rb)]\\
&=\log\specrad[(f_T)_*:H_*(M; \Rb)\to H_*(M; \Rb)]\\
&=\log\specrad[(f_T)_*:H_2(M; \Rb)\to H_2(M; \Rb)]\\
&=\log|\lambda_1\lambda_2|=\log|\lambda_1|+\log|\lambda_2|.
\end{array}
\end{equation}
This together with \eqref{eq:ftentropy} implies Property 2 of Theorem \ref{theorem:main1}.
\end{proof}

\para{Is the homological lower bound sharp?} 
There are infinitely many $T\in\SL(4,\Zb)$ (even up to conjugacy and taking powers) having $3$ eigenvalues $\lambda_1, \lambda_2, \lambda_3$ lying outside the unit circle. This holds for example when $T$ is the companion matrix of $x^4+ax+1$ for any integer $a\geq 3$.  In this 
case something interesting happens: labeling the eigenvalues 
so that with $|\lambda_1|\geq |\lambda_2|\geq |\lambda_3|>1$, Theorem \ref{theorem:main1} gives that 
\[\htop(f_T)=\log|\lambda_1|+ \log|\lambda_2|+\log|\lambda_3|\]
which is strictly greater than the homological lower bound $\log|\lambda_1|+\log|\lambda_2|$ given in \eqref{eq:hbound}. The point here is 
the following.  On the one hand, the computation of 
$\htop(f_T)$ given in \S\ref{subsection:htopmain} below picks up 
$\log|\lambda_1\lambda_2\lambda_3|$ coming from the action of $T$ on 
$H_3(A)$, where $A$ is the $4$-torus in the Kummer construction (see \S\ref{section:construction}).   On the other hand, $H_3(M)=0$, and the only nontrivial homological data about $f_T\in\Diff(M)$ is its action on $H_2(M)$, which records only the products of {\em pairs} of eigenvalues.  

As far as we can tell, the only known lower bounds for $\htop$ of an isotopy class of diffeomorphism on an $M$ with $\pi_1(M)=0$ come from 
its induced action on homology  as above (i.e.\ Yomdin's Theorem).  In his 2006 ICM talk, Shub (\cite{Sh}, \S 4.2) asks whether this homological lower bound is always realized, or if there is at least a sequence $f_n\in\Diff(M)$ in the given  isotopy class converging to this lower bound.  
We make the following conjecture, which would give a negative answer to Shub's questions.

\begin{conjecture}[{\bf Homologically undetectable entropy}]
\label{conjecture:egap}
Let $T\in\SL(4,\Zb)$ have exactly $3$ eigenvalues $\lambda_1, \lambda_2,\lambda_3$ lying outside the unit circle, ordered so that $|\lambda_1|\geq |\lambda_2|\geq |\lambda_3|>1$.  Then the minimal entropy of any $g\in\Diff(M)$ isotopic to $f_T$ satisfies 
\[\htop(g)>\log\specrad[g_*:H_*(M; \Rb)\to H_*(M; \Rb)]\]
with the difference between the two sides being close to (equal to?) 
$\log|\lambda_3|$.
\end{conjecture}

Note that the linear diffeomorphism $g_T$ on the $4$-torus $A$ induced by $T$ does satisfy 
\begin{equation}
\label{eq:realize3}
\htop(g_T)=\log|\lambda_1|+ \log|\lambda_2|+\log|\lambda_3|=\log\specrad[(g_T)_*:H_*(M; \Rb)\to H_*(M; \Rb)]
\end{equation}
and so $g_T$ minimizes entropy in its homotopy class.  The reason that 
\eqref{eq:realize3} holds is that for the action of 
$(g_T)_*$ picks up all $3$ eigenvalues 
$\lambda_1,\lambda_2,\lambda_3$ on $H_3(A)$.

\para{Relation with other work.} Theorem \ref{theorem:main1} addresses the following question, which is implicit in work of Smale, Shub, Shub-Sullivan and others.

\begin{question}
Does every nontrivial isotopy class of diffeomorphism on a smooth manifold have a representative minimizing entropy in the isotopy class? If so, find such a representative.
\end{question}

As we indicate above, Theorem \ref{theorem:main1} seems to give 
(as far as we can tell) the first new examples since the 1980s of an entropy minimizer.

We also mention that there exist positive entropy biholomorphic automorphisms of K3 surfaces whose dynamics is quite different than what appears on Kummer surfaces; see McMullen \cite{Mc} and Cantat \cite{Ca}.

This paper fits into a program of the authors to give a ``Thurston classification'' for elements of the mapping class 
group of the K3 manifold.  This paper is the ``pseudo-Anosov'' case.  See \cite{FL1} and \cite{FL2} for the finite order and reducible cases, respectively.

\para{Acknowledgements.} We thank Amie Wilkinson for helpful discussions and Dan Margalit for numerous corrections.  We are extremely grateful to Serge Cantat and Curt McMullen for extensive, insightful comments and corrections on an earlier draft of this paper.

\section{Constructing the diffeomorphism $f_T$}
\label{section:construction}

In this section we construct, for each $T\in\SL(4,\Zb)$, 
a diffeomorphism $f_T\in\Diff(M)$ of the K3 manifold $M$.  In \S\ref{subsection:firstprop} and 
\S\ref{section:dynamical} we will prove that $f_T$ has the properties claimed in the statement of Theorem \ref{theorem:main1}.  
The construction will involve a path from the  identity to $T$ in the ambient Lie group $\SL(4, \Rb)$; 
since the  fundamental group of the  latter has order $2$, there are two homotopy classes of such paths.

Throughout this section, let $V$ denote a real, $4$-dimensional vector space and let $L\subset V$ be a lattice in $V$.

\subsection{The Kummer surface}
\label{subsection:Kummer}
The quotient $A:=V/L$ is a real torus of  dimension  $4$, and is a group under addition mod $L$.  Its subgroup $A[2]$ of $2$-torsion points can be identified with $\frac{1}{2}L/L\cong \Fb_2\otimes L$, and so $A[2]$ has the structure of  a $\Fb_2$-vector space of dimension $4$. 

Assume now that $V$ has a complex structure. Then $A$ is a complex surface and one can define the complex blowup 
\[
\Bl_{A[2]}(A)\to A
\]
of $A[2]$ in $A$. Thus $\Bl_{A[2]}(A)$ has sixteen exceptional divisors, each a copy of $\Cb\Pb^1$ with self-intersection $-1$.  Since $-{\rm I}\in \GL(L)$ is $\Cb$-linear, its  associated involution of $A$  is holomorphic and lifts to a holomorphic involution 
\[\iota:\Bl_{A[2]}(A)\to \Bl_{A[2]}(A).\]
Although $\iota$ does not act freely on $A$ (it fixes the exceptional set pointwise), the orbit space of $\iota$ is a nonsingular complex surface, the {\em Kummer surface} 
\[
\Kum(A):=A/\la\iota\ra
\]  
associated with $A$.  If $A$ is projective then so is $\Kum(A)$.  The complex manifold $\Kum(A)$ is a particular type of K3 surface.   We refer to a K3 manifold of this type as a  {\em Kummer manifold}.

\subsection{The case of complex automorphisms $T$}\label{subsect:complexcase}

Since $V$ can be identified with  $\Rb\otimes L$, there is a representation 
\[\GL(L)\to \Aut(A)\]
denoted by $T\mapsto T_A$.  The restriction of this map 
to $A[2]$, denoted $T\mapsto T_A[2]$, factors through an action 
of $\GL(\Fb_2\otimes L)$.
Note that  $A[2]$ is the fixed-point set of the involution  of $A$ induced by $-{\rm I}\in \SL(L)$. The group $\Fb_2\otimes L$ also acts on $A$ as a translation group.  Together with the $\GL(L)$-action this gives an action of  the semi-direct product $(\Fb_2\otimes L)\rtimes\GL(L)$ on $A[2]$.

Let $T\in\SL(L)\cong\SL(4,\Zb)$ be given.  If $T$ 
leaves invariant complex structure on $V$ (relative to which $T$ then becomes complex-linear), then the induced diffeomorphism $T_A:A\to A$ is biholomorphic, and so preserves the set of complex lines in the tangent spaces of $A[2]$.  It follows that $T_A$ lifts to an automorphism 
$\tilde{T}_A$ of $\Bl_{A[2]}(A)$.  Since $T$ commutes with $-{\rm I}\in \SL(L)$ it follows that $\tilde{T}_A$ commutes with $\iota$, so that 
the automorphism $\tilde{T}_A$ descends to a biholomorphic automorphism of $\Kum(A)$.  We let $f_T$ denote this biholomorphic automorphism.  

These automorphisms are well-studied, see e.g.\  \cite{Mc}.  However, most $T\in\SL(L)$ are not complex-linear.  Such $T$ 
do not induce in an obvious way a map 
$\Bl_{A[2]}(A)\to \Bl_{A[2]}(A)$, and so the construction just described 
fails.  The four-step construction in the next subsection is meant to repair this.   

\subsection{The construction of $f_T$}
\label{subsection:construct}
We first need the notion of real oriented blow-up.

\begin{definition}[{\bf Real oriented blow-up}]
Let $N$ be an $n$-manifold and $p\in N$. Then the \emph{real, oriented blow-up}  is an $n$-manifold  $\RBl_p(N)$ with boundary $\partial \RBl_p(N)$.  It comes with a differentiable map $\pi: \RBl_p(N)\to N$ and is specified by the following properties:
\begin{enumerate}
\item $\pi$ is a diffeomorphism over $N\ssm \{p\}$; and 
\item $\pi^{-1}(p)=\partial \RBl_p(N)$, and can be identified  with 
the $(n-1)$-sphere of rays in $T_pN$.   Over  a coordinate neighborhood of $p$ the map $\pi:  [0, 1)\times S^{n-1}\to\Rb^n$ is given by 
$\pi(t, u):=tu$.
\end{enumerate}
\end{definition}

See Figure \ref{figure:rblow1}.  
\begin{figure}[h]
\begin{center}
\includegraphics[scale=0.24]{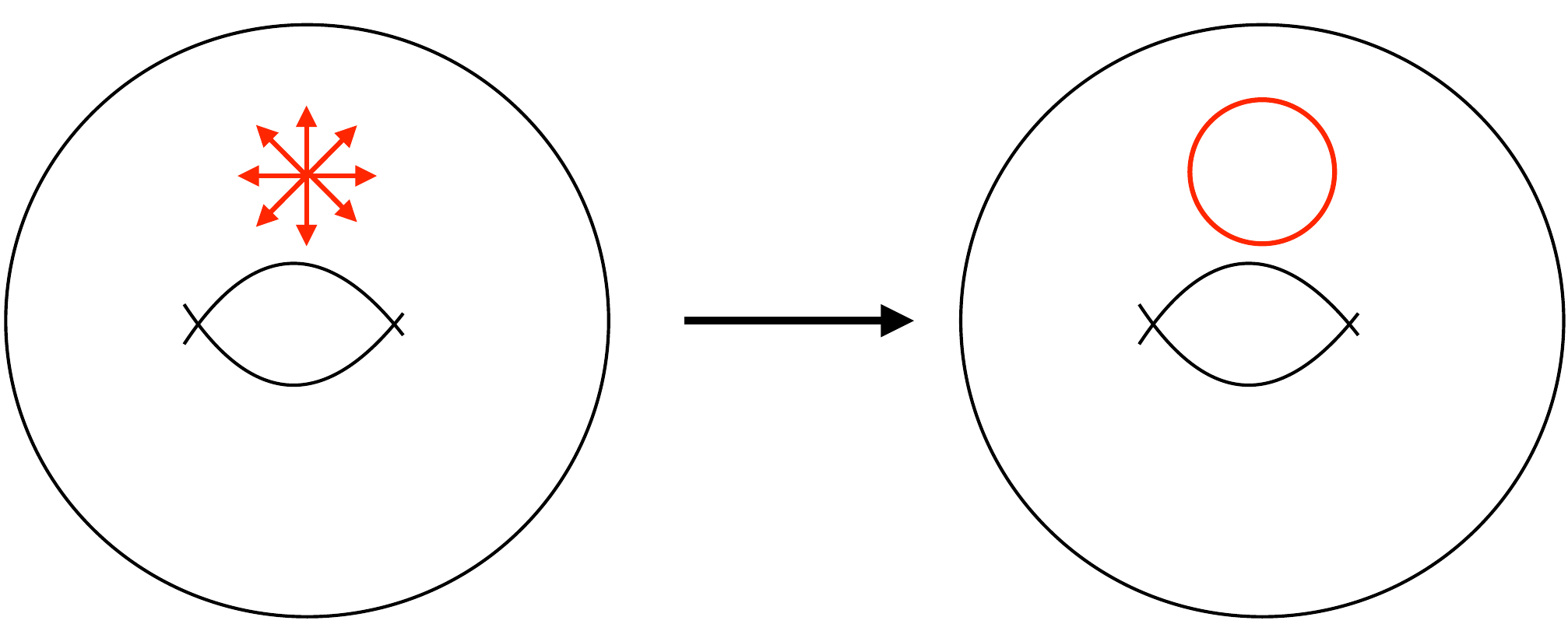}
\end{center}
\caption{\small The real-oriented blowup on a $2$-dimensional manifold.  The result is a $2$-manifold with one boundary component. }
\label{figure:rblow1}
\end{figure}

It is clear  that every differentiable curve $\gamma: [0, \epsilon)\to N$ with $\gamma(0)=p$ and $\gamma' (0)\not=0$ lifts uniquely to a differentiable curve $\tilde\gamma: [0, \epsilon)\to\RBl_p(N)$ whose initial point only depends on the ray in $T_pN$ spanned by
$\gamma' (0)$. 

The real-oriented blowup $\RBl_p(N)$ is canonical in the sense that every diffeomorphism of $N$ that fixes $p$ lifts to a diffeomorphism of $\RBl_p(N)$.   More generally, if $P\subset N$ is a discrete set , then one can define 
$\RBl_{P}(N)\to N$, and if $f\in\Diff(N)$ preserves $P$, then 
then $f$  lifts to a diffeomorphism $\RBl_{P}(N)\to \RBl_{P}(N)$.
 
If $N$ is $4$-dimensional and comes with  complex structure at $p\in N$, then the  real oriented blowup $\RBl_p(N)\to N$ factors through  the complex blowup $\Bl_p(N)\to N$ via the  Hopf map  which assigns to a real ray in the tangent space $T_pN$ the complex line spanned by it (in  coordinates,  $S^3\to \CP^1\cong S^2$). We call this procedure the {\em Hopf collapse}.

We will need the following lemma.

\begin{lemma}\label{lemma:fundgrp}
The inclusion of $\GL(2,\Cb)$ into the identity component $\GL(4, \Rb)^+$ of $\GL(4, \Rb)$ induces a surjection on fundamental groups.
\end{lemma}
\begin{proof}
The Gram-Schmidt  process  shows that the inclusions  
${\rm U}(2)\subset \GL(2,\Cb)$ and 
 $\SO(4)\subset\SL(4, \Rb)\subset \GL(4, \Rb)^+$ are deformation retracts, and hence it suffices to prove that the inclusion ${\rm U}(2)\subset \SO(4)$ 
induces a surjection on fundamental groups. The  group $\SO(4)$ is 
doubly covered by $\SU(2)\times \SU(2)$ with the covering  involution 
defined by $(-{\rm I},-{\rm I})$ and such that the preimage of $\U(2)$ is 
$\SU(2)\times \U(1)$, where $\U(1)$ is embedded in $\SU(2)$ as the diagonal group:
$u\mapsto 
(\begin{smallmatrix}
u & 0\\
0 & \bar u
\end{smallmatrix}) $. Since $\SU(2)\times U(1)$ contains the covering involution $(-1,-1)$, the lemma follows.
\end{proof}

With the above in hand, we are ready to construct, for each $T\in\SL(L)$, a diffeomorphism $f:\Kum(A)\to\Kum(A)$. 

\para{Step 1 (Lift to the real-oriented blowup):} 
Going back to the torus $A$ above, let 
\[
Y:=\RBl_{A[2]}(A)\to A
\]
be the real oriented blowup of $A$ at the $16$ points $A[2]$.  See Figure \ref{figure:4dblow}.  

\begin{figure}[h]
\begin{center}
\includegraphics[scale=0.24]{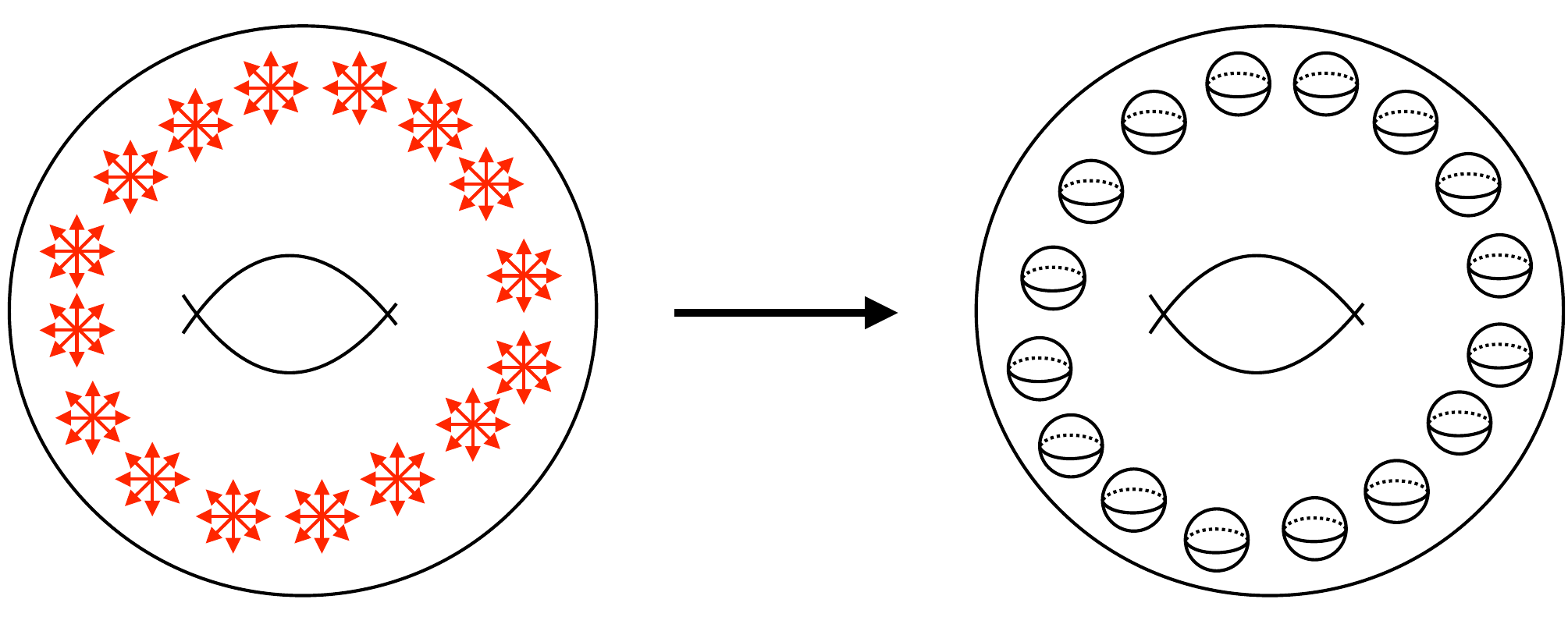}
\end{center}
\caption{\small Performing a real-oriented blowup at the sixteen $2$-torsion points $A[2]$ of the $4$-dimensional torus $A$.  The result is a 
compact manifold $Y:=\RBl_{A[2]}(A)\to A$ with sixteen boundary components, each diffeomorphic to a $3$-sphere. }
\label{figure:4dblow}
\end{figure}

Observe that $Y$ is a compact, oriented $4$-manifold  which comes with an action of 
the semi-direct product $L[2]\rtimes\SL(L)$, where $L[2]$ acts on $A$ by translations.  The boundary $\partial Y$ of $Y$ 
has $16$ connected  components, naturally labeled by 
elements of $A[2]$; so $\partial Y=\coprod_{\alpha\in A[2]}\partial_\alpha Y$.  The components of $\partial Y$ are 
permuted simply transitively by $\Fb_2\otimes L$ (acting as a translation group), and each component $\partial_\alpha Y$ 
is diffeomorphic to $S^3$.  Since $T_A$ is a diffeomorphism it lifts to a  diffeomorphism
\[
T_Y :Y\to Y.
\]
that permutes the $16$ components of $\partial Y$ according to the induced action of $T$ on $A[2]$.
If  $T\in\GL(L)$ is congruent to the identity modulo $2$, the diffeomorphism $T_Y$ fixes each boundary component $\partial_\alpha Y\cong S^3$, and acts on it as at the antipodal map.  

\para{Step 2 (Extension to a collar neighborhood):}  Let $Y'$ be obtained by glueing on  $Y$ 
an {\em external collar} along its boundary,  by which we mean that we add a product 
\[[0,1]\times \partial Y\cong [0,1]\times L[2]\times S^3\] and identify $\{0\}\times \partial  Y$  with $\partial Y\subset Y$ in the obvious manner.  This  $Y'$ is diffeomorphic to $Y$ and the inclusion $Y\subset Y'$ is a strong $L[2]\rtimes \SL(L)$-equivariant deformation retract.  The action of $L[2]\rtimes \SL(L)$ 
on $Y$ extends to $Y'$ in the obvious way.  See Figure \ref{figure:collar}.

\begin{figure}[h]
\begin{center}
\includegraphics[scale=0.3]{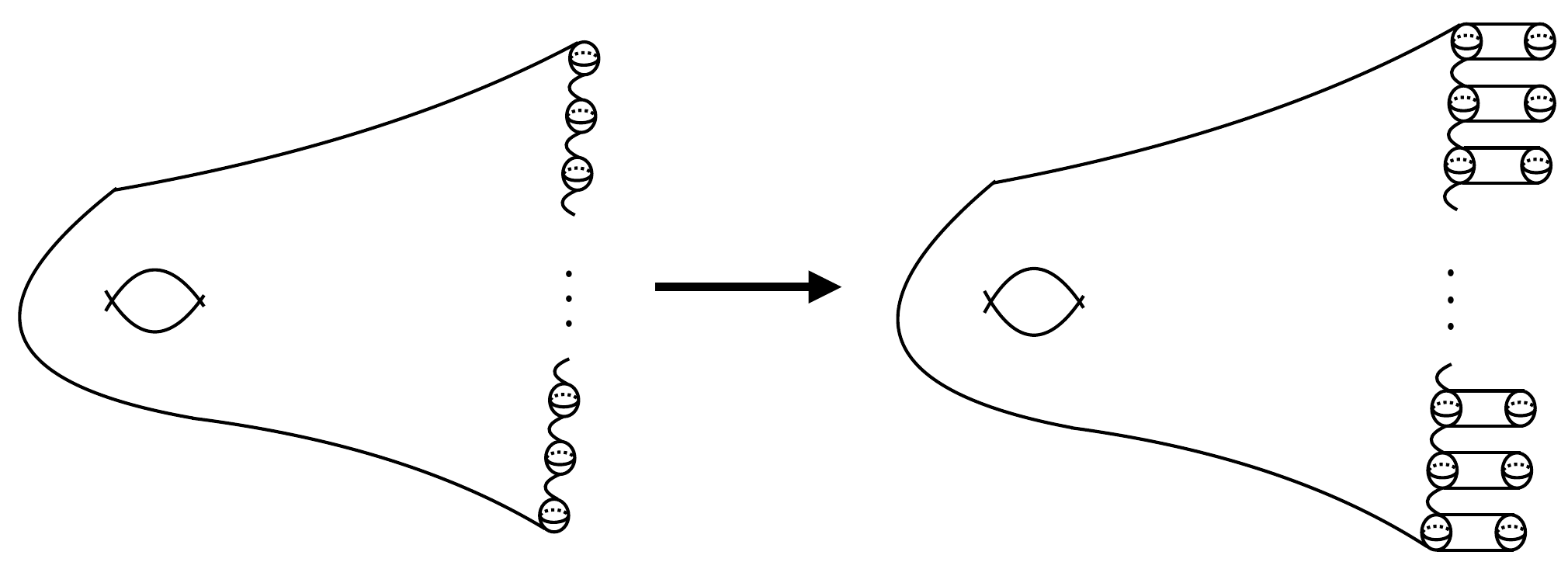}
\end{center}
\caption{\small Adding a collar to each of the $16$ boundary components 
of $Y$ to obtain the diffeomorphic manifold $Y'$. Note that the manifold $Y$ in this figure is the same as (but a different visualization of) 
the manifold $Y$ in Figure \ref{figure:4dblow}.}
\label{figure:collar}
\end{figure}

Fix a complex structure $J$ on $V$, so that we have defined the  group 
$\GL(V,J)$.   Choose a path $\gamma_T:[0,1]\to \SL(V)$ from  $T$ to  an element of $\GL(V,J)$ that is constant near $0$ and near $1$ (\footnote{We emphasize that this is the exact place where we assume that $T\in\SL(L)$ and not just in $\GL(L)$.}). Lemma \ref{lemma:fundgrp} implies that  the space of all such paths is arcwise connected. This is important in what follows.
We then extend $T_Y$ in an $L[2]\rtimes \SL(L)$-equivariant manner to a diffeomorphism 
\[T_{Y'}:Y'\to Y'\]
  by setting 
$T_{Y'}|Y:=T_{Y}$  and defining $T_{Y'}$ on each 
$[0,1]\times S^3$ collar about each component of $\partial Y$ by executing $\gamma_T(t)$ on $\{t\}\times S^3$.  More precisely, define $T_{Y'}$ on $[0,1]\times \partial Y\cong [0,1]\times A[2]\times  \partial_0Y$ to be :
\[
T_{Y'}(t, \alpha, z):=(t, T[2](\alpha),\gamma_T(t)(z))
\]
for all $(t, \alpha, z)\in [0,1]\times  A[2]\times \partial_0 Y$.  Note that 
$T_{Y'}$ is a diffeomorphism; it acts acts on a neighborhood of $\partial Y'$ through its action on $A[2]$; and it acts as the identity near $\partial Y'$.

\para{Step 3 (Hopf collapse):} 
Performing the Hopf collapse on each of the $16$ boundary components of $Y'$ gives a smooth map $h:Y'\to \overline Y'$ to a closed manifold $\overline Y'$.  It is straightforward to check that this manifold is diffeomorphic to the complex blowup $\Bl_{A[2]}(A)$, where the image under $h$ of each boundary component of $Y'$ is an exceptional divisor in 
$\overline{Y}'\cong\Bl_{A[2]}(A)$.  Since, as noted in Step 2, the diffeomorphism $T_{Y'}$ acts (up to permutation) as the identity on a neighborhood of  the  boundary of $Y'$, it follows in particular that $T_{Y'}$ respects the fibers of the Hopf collapse and so descends to a diffeomorphism 
\[\
T_{\overline Y'}: \overline Y'\to \overline Y'.
\]
The action of $T_{\overline Y'}$ on $16$ exceptional divisors of $T_{\overline Y'}$ is still  through its action on $A[2]$.

\para{Step 4 (The involution):}
The transformation $-{\rm I}\in\SL(L)$ is complex-linear  and defines  an involution $\iota_{\overline Y}$ on $Y$. This extends to an involution  $\iota_{\overline Y}$ of  $\overline Y$ (on the collar it will be product of $1_{[0,1]}\times \iota_Y$), which  in turn descends to an involution 
$\iota_{\overline Y'}$ of $\overline Y'$.   Let 
\[
M:=\overline Y'/\la \iota_{\overline Y'}\ra
\]
be the orbit space.   Note that $M$ is diffeomorphic with the Kummer manifold $\Kum(A)$ and hence is a  K3 manifold. Since $T_{\overline Y'}$ commutes with $\iota_{\overline Y'}$ it induces a diffeomorphism 
\[f_T:M\to M.\]
This is the claimed diffeomorphism.

\begin{remark}[{\bf Naturality}]
\label{remark:naturality}
The diffeomorphism $f_T:M\to M$ thus obtained depends on the choice of the path 
$\gamma_T$ in $\GL(V)^+$ from  the identity to an element of 
$\GL(V, J)$,  which we 
assumed to be constant near its end points. 
Another choice $\gamma'_T$  produces 
another diffeomorphism $f'_T$. However, by Lemma \ref{lemma:fundgrp} 
$\gamma'_T$ and $\gamma_T$ will be in the same arcwise connected component of such paths. This implies that  $f'_T:M\to M$
is isotopic with $f_T$. It follows that  our construction defines an embedding of $\SL(L)$ into the mapping class group $\Mod(M)$ of $M$.
\end{remark}

\begin{remark}[{\bf Isotopic but nonconjugate entropy minimizers}]
\label{remark:nonconjugate}
We remark that varying the choice of path $\gamma_T\subset\GL(V)^+$ gives uncountably many diffeomorphisms $f_T$, each an entropy minimizer, each isotopic to the others, but no two conjugate in $\Diff(M)$.  
To see this, first note that smooth conjugacy leaves invariant hyperbolic sets, and so leaves invariant the complement of the hyperbolic set. In our case the latter is a neighborhood of sixteen $2$-spheres in $M$.  

Now for any $n\geq 1$ one can choose the path $\gamma_T$ in Step 2 of the construction of $f_T$ to be a path $\beta_n\subset\GL(V)^+$ that is the identity matrix on precisely $n$ time intervals.  Even though smooth conjugacy does not {\it a priori} preserve the ``time direction'' on the (quotient of) the collar around a $2$-sphere, the number $n$ is a conjugacy invariant because it is the minimal number of distinct path components on which the diffeomorphism is the identity map.  

This gives countably many nonconjugate $f_T$. A slightly more elaborate construction gives uncountably many.
\end{remark}

\subsection{First properties of $f_T$}
\label{subsection:firstprop}

In this section we verify Properties 3 and 4 of Theorem \ref{theorem:main1}.  We assume here that $L$ has been oriented, which means that we have singled out a generator of $\wedge^4 L$. This determines a quadratic form $\delta: \wedge^2 L\to \wedge^4L\cong \Zb$ defined by $\delta(\alpha)=\det(\alpha\wedge\alpha)$. It is nondegenerate and has signature $(3,3)$. We use the same symbol for its $\Rb$-linear extension to
$\wedge^2_\Rb V$. It is a classical fact that any isotropic line in  $\wedge^2_\Rb V$ is the $\wedge^2(P)$ for some $2$-plane $P\subset V$.

\para{Proof of Property 3 of Theorem \ref{theorem:main1}: }
We proved in \S\ref{subsect:complexcase} that if $T$ preserves a complex structure on $V$, that is if there exists $J\in \SL(V)$ with $J^2=-1$ and $JT=TJ$, then $f_T$ is a biholomorphic automorphism of the associated Kummer surface.  In that case the $+\sqrt{-1}$-eigenspace $V^{1,0}(J)$ of $J$ in $V_\Cb$ has complex dimension $2$ and $V^{2,0}(J):=\wedge^2_{\Cb}V^{1,0}(J) $ 
is a complex  line in  $(\wedge^2 L)_\Cb$, invariant under $T$. This line is isotropic  for the  $\Cb$-bilinear extension of $\delta$ and positive-definite for its hermitian extension.  This accounts for the summand $H^{2,0}(M)$ in the Hodge decomposition of $H^2(M; \Cb)$.

To see the converse, first note that the Kummer construction produces an embedding of $\wedge^2 L\to H^2(M; \Zb)$ for which 
\[
T\in \SL(L)\mapsto (f_{T}^*)^{-1}\in \Orth(H_M)
\]
 is equivariant. If  $f_T$ respects a complex structure of $M$ then $f^*_T$ preserves the associated Hodge decomposition. The subspace  
 $H^{2,0}(M)\oplus H^{0,2}(M)$ is the complexification of a real, oriented,  positive-definite plane $\Pi\subset H_M\otimes \Rb$ that must be  a sum of eigenspaces of $f_T^*$.
 Since  $\wedge^2V$ is $f^*_T$-invariant and has negative-definite orthogonal complement,  $\Pi$ projects isomorphically 
 onto a positive-definite oriented plane $\Pi'$ in  $\wedge^2V$ that is also $f^*_T$-invariant. Such  a plane  $\Pi'$ defines a Hodge structure on 
$\wedge^2V$ (with  $(\wedge^2V)^{2,0}$ being the orthogonal projection of  $H^{2,0}(M)$ in $\wedge^2V_\Cb$). It is well-known (and easy to show) that  such  a Hodge 
structure  comes from a complex structure $J$ on $V$ (with the $\wedge_\Cb^2$ of its $\sqrt{-1}$ eigenspace yielding $(\wedge^2V)^{2,0}$) and that $J$  is unique up to sign. 
This uniqueness implies that $T$ will commute with $J$. Since $T$ has four distinct real eigenvalues,  the corresponding eigenlines are preserved by $J$, a contradiction. Thus $f_T$ cannot respect any complex structure on $M$.  
 
\para{Proof of Property 4 of Theorem \ref{theorem:main1}: }
When $T\in\SL(L)$ is {\em Anosov}, that is, when $T$ has all eigenvalues off the unit circle, then the eigenvalues inside (resp.\ outside)  
the unit circle determine a $T$-invariant decomposition of $V$ into two nonzero subspaces: a subspace  
$V_-$ on which $T$ is contracting and a subspace $V_+$  on which it is expanding. The images of the translates of $V_{\pm}$ 
under the  projection  $V\to A$  give two complementary $T_A$-invariant foliations on $A$.  Each of these foliations comes equipped with a transverse measure given by integrating against the closed differential form defining it, and $T_A$ leaves invariant the projective class of each of these measures.  The diffeomorphism $T_A$ is a linear Anosov diffeomorphism.

As noted in Step 1 of \S\ref{subsection:construct}, $T_A$ extends to a diffeomorphism $T_Y$ of the real-oriented blowup $Y:=\RBl_{A[2]}(A)$.  A $2$-dimensional visualization of the corresponding $T_Y$-invariant foliations on $Y$ is given in Figure \ref{figure:open}.

\begin{figure}[h]
     \centering
     \begin{subfigure}{0.45\textwidth}
         \centering
        \includegraphics[scale=0.2]{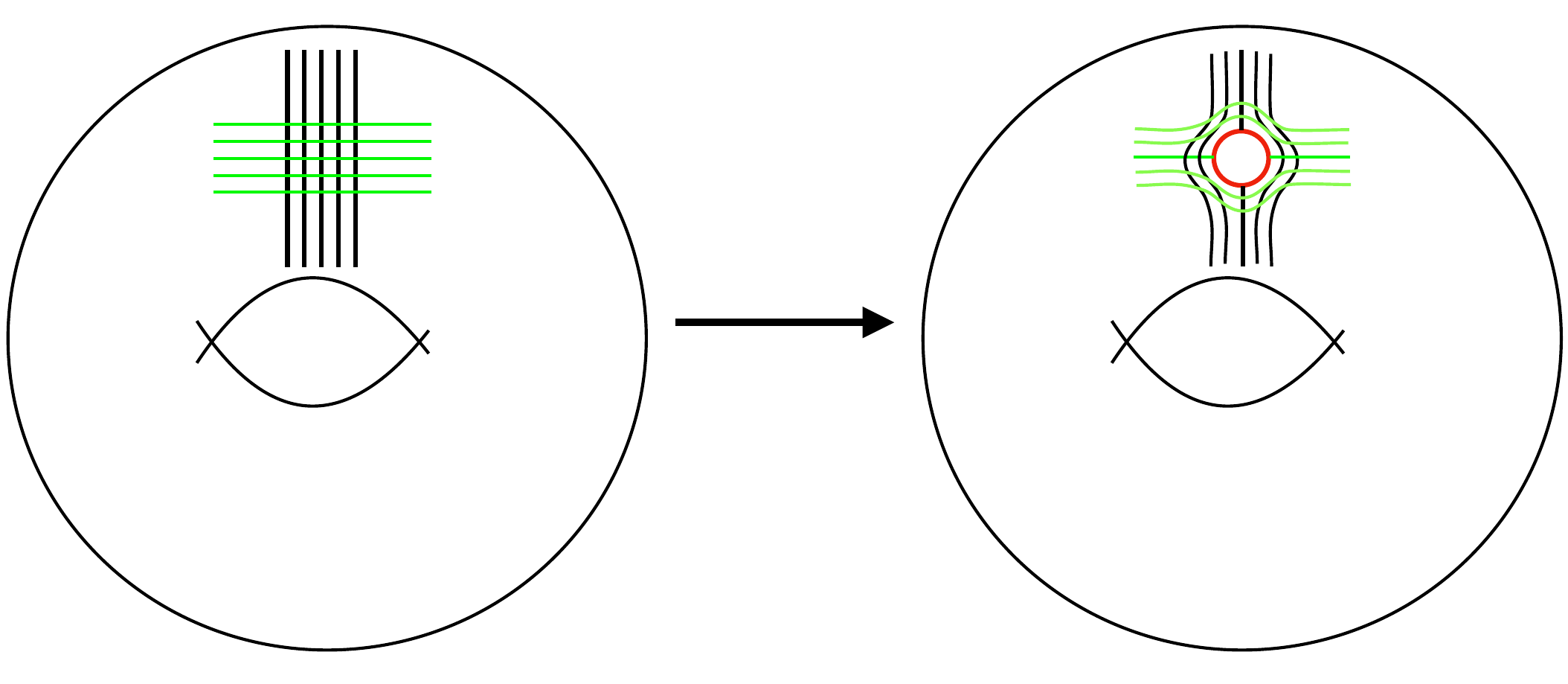}
         \caption{The effect of real-oriented blow-up on a pair of transverse foliations.}
       
     \end{subfigure}
     \begin{subfigure}{0.45\textwidth}
         \centering
         \includegraphics[scale=0.2]{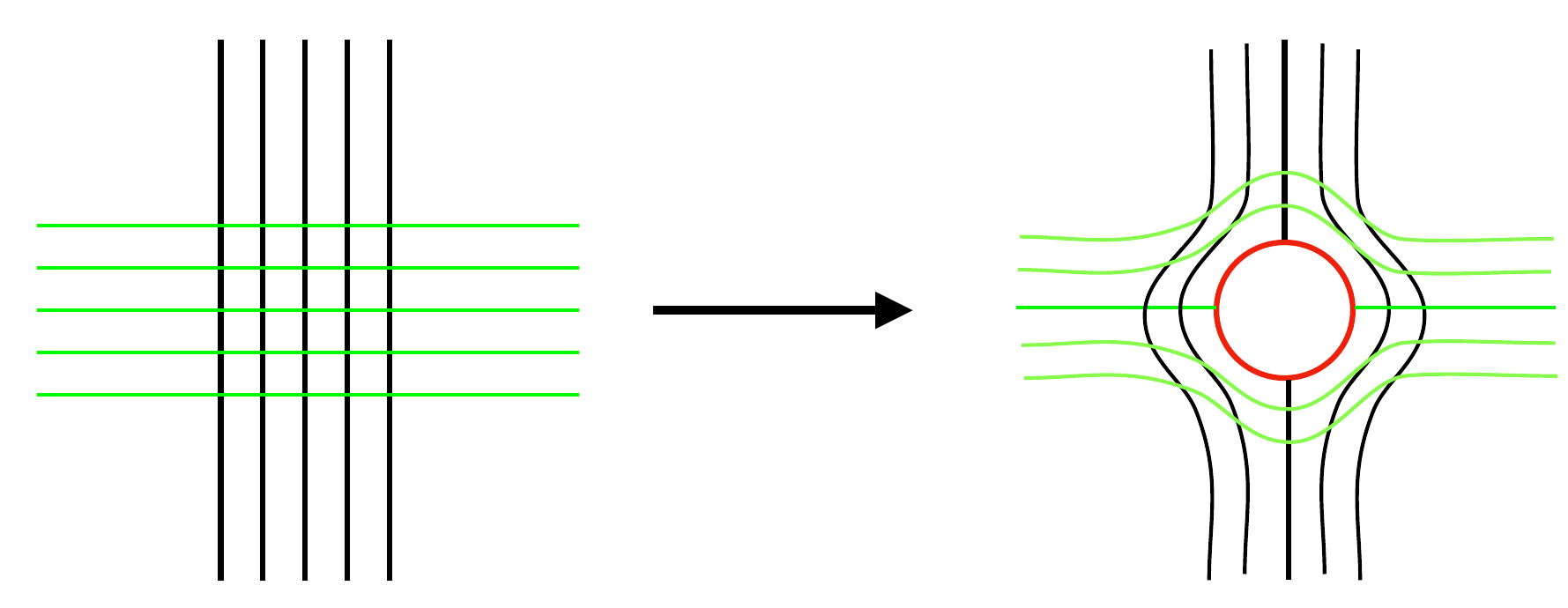}
         \caption{A zoom-in of (a).}
              \end{subfigure}
     \caption{}
     \label{figure:open}
\end{figure}

By construction, $T_Y$ preserves a neighborhood of $\partial Y$ (which is a union of sixteen $3$-spheres), and is Anosov in the complement of this neighborhood.  In fact more is true: by choosing the speed of the path $\gamma_T\subset\SL(V)$ one can make this neighborhood as small as one wants in that it can be taken to lie inside any given sufficiently 
small neighborhood of $\partial Y$.  These properties survive the Hopf collapse and the involution.  We note that all the properties of $f_T$ claimed in Theorem \ref{theorem:main1} hold for $f_T$ independently of the choice of the path $\gamma_T$ used to construct it.  

The previous paragraph shows that $f$ is what we call  {\em of pseudo-Anosov type}: for any neighborhood $N$ of the union of the sixteen $(-2)$ rational curves in $M=\Kum(A)$ there is a choice of $\gamma_T$ so that $f_T$ is Anosov in $M-N$.  

\subsection{Homological spectral radius of $f_T$}

In this subsection we compute the spectral radius of the action 
induced by $f_T:M\to M$ on $H_2(M)$.  Recall that the {\em spectral radius $\specrad(A)$} of a linear transformation $A$ of a finite-dimensional vector space is the maximal absolute eigenvalue of $A$.

\begin{proposition}[The spectral radius of $f_{T*}$]
\label{proposition:specrad}
Let $M$ be the K3 manifold.  Let $T\in\SL(L)$, and denote the eigenvalues of $T$ by $\lambda_1,\ldots,\lambda_4$.  Let $f_T\in\Diff(M)$ denote the diffeomorphism constructed in \S\ref{subsection:construct} above.  The induced action $f_{T*}:H_2(M)\to H_2(M)$ satisfies
\[\specrad(f_{T*})=\max\{|\lambda_i\lambda_j|:1\le i<j\le 4\}.\] 
\end{proposition}

\begin{proof}
For simplicity let $H_i(M):=H_i(M;\Zb)$.  Since $M$ is a simply-connected, 
closed, oriented $4$-manifold, it follows from Poincar\'{e} Duality and the Universal Coefficients Theorem that $H_1(M)=H_3(M)=0$ and $H_0(M)\cong H_4(M)\cong\Zb$. Note that $f_T$ is orientation-preserving by construction (\footnote{In fact any homeomorphism of $M$ is orientation-preserving, but we will not need this.}).  Thus 
\[
\specrad(f_{T*})=\specrad(f_{T*}|_{H_2(M)}).
\]  

The images in $\Kum(A)$ of the $16$ exceptional divisors in $\Bl_{A[2]}(A)$ are $(-2)$ vectors; they generate a rank $16$, negative-definite sublattice $W\subset H_2(M)$.  The orthogonal complement of this sublattice has rank $6$, signature $(3,3)$, and contains with finite index a copy of $\wedge^2L$ coming from the classes of $H_2(A)$; see for example \S 3 of \cite{LP}.  The rank $22$ sublattice 
$W\oplus \wedge^2L$ of $H_2(M)$ has finite index in $H_2(M)$, and so 
its rational span in $H_2(M;\Qb)$ is all of $H_2(M;\Qb)$; that is:
\begin{equation}
\label{eq:sum1}
H_2(M;\Qb)=(W\otimes\Qb)\oplus \wedge^2(L\otimes\Qb).
\end{equation}

Of course the eigenvalues of $(f_T)_*$ on $H_2(M)$ are equal to those of $(f_T)_*$ on $H_2(M;\Qb)$, which we now compute.   

Any $T\in\SL(L)$ acts as a permutation on the set $L[2]$ of $2$-torsion points, hence acts as a permutation on the $16$ exceptional divisors of $\Bl_{A[2]}(A)$, hence as a permutation representation on $W$.  It follows that the eigenvalues of the restriction of $(f_T)_*$ to the invariant 
summand $W$ of $H_2(M;\Qb)$ have absolute value equal to $1$.  On the other hand, since $(f_T)_*$ acts on the other invariant summand 
$\wedge^2(L\otimes \Qb)$ via the representation $\wedge^2(L\otimes\Q)$ of $T$, it follows that 
its eigenvalues on this summand are precisely $|\lambda_i\lambda_j|$ with $1\le i<j\le 4$.  Thus 
\[\specrad(f_{T*})=\max\{|\lambda_i\lambda_j|:1\le i<j\le 4\}\]
as asserted earlier.
\end{proof}

Proposition \ref{proposition:specrad} gives a lower bound on the entropy of any $g\in\Diff(M)$ homotopic to $f_T$.

\begin{corollary}
\label{corollary:Yomdin}
Let $T\in\SL(L)$ have eigenvalues $\lambda_i$ ordered so that 
$|\lambda_1|\geq \cdots \geq |\lambda_4|$.  Let $f_T\in\Diff(M)$ be as above.  For any $g\in\Diff(M)$ homotopic to $f_T$ the following lower bound holds:
\[\htop(g)\geq \log|\lambda_1|+\log|\lambda_2|.\]
\end{corollary}

\begin{proof}
Yomdin \cite{Y} proved that any $g\in\Diff(M)$ satisfies Shub's Entropy Conjecture: 
\[
\htop(g)\geq \log\specrad(g_*:H_*(M)\to H_*(M))
\]
Since $g$ is homotopic to $f_T$ it follows that $g_*=f_{T*} :H_2(M)\to H_2(M)$. The corollary now follows from Proposition \ref{proposition:specrad}.
\end{proof}

\section{Computing the entropy of $f_T$}
\label{section:dynamical}

The goal of this section is to establish Property 1 of Theorem \ref{theorem:main1}

\subsection{General results on entropy}
\label{subsection:entropy}

In this subsection we state some of the foundational results on entropy that we will later need.  The main references we use are Bowen \cite{B} and Adler-Konheim-McAndrew \cite{AKM} 

It will be convenient to use the characterization of entropy for compact metric spaces as given in  \S 1 of Bowen \cite{B}. It is defined as follows. Let $T:X\to X$ be a uniformly continuous map on a compact metric space $(X,d)$.  Let $K\subset X$ be compact.   A subset $E\subseteq K$ is said to {\em $(n,\epsilon)$-span} $K$ (with respect to $T$) if for each  $x\in K$ there exists $y\in E$ so that 
$d(T^j(x),T^j(y)\leq \epsilon$ for all $0\leq j< n$. Let $r_n(\epsilon,K)$ be the smallest cardinality of any set $(n,\epsilon)$-spanning $K$.  Define the 
{\em entropy} $h_d(T,K)$ of $T$ relative to $K$ to be
\[
h_d(T,K):=\lim_{\epsilon\to 0}\ (\ \limsup_{n\to\infty}\tfrac{1}{n}\log r_n(\epsilon,K)\ ).
\]
The entropy $h_d(T)$ on $X$ is defined as
\[
h_d(T):=\sup_{K\subseteq X \ \text{compact}}h_d(T,K).\]
Bowen proves in \cite{B} results on the entropy $h_d(T)$ of self-maps of metric spaces $(M,d)$.  For compact manifolds, which is all we consider at each stage, this entropy  coincides with the  topological  entropy; see the remark after Proposition 3 on page 403 of \cite{B}.  That is:
\[\htop(T)=h_d(T).\]

\begin{lemma}
\label{lemma:relative1}
 Let $T:X\to X$ be a uniformly continuous map on a compact metric space $(X,d)$.  Let $K\subset X$ be compact.   If $K$ is invariant then 
 \begin{equation}
\label{eq:relative1}
h_d(T,K)=h_d(T|K).
\end{equation}
\end{lemma}

\begin{proof}
First note that the definition of $h_d(T,K)$ depends only on the $T$-orbits of points in $K$.  Since $K$ is $T$-invariant, each such $T$-orbit lies in $K$, 
so that $h_d(T,K)=h_d(T|K,K)$.  As remarked at the top of Page 403 of \cite{B}, for any compact metric space $(Y,d)$ it holds that 
$h_d(T,Y)=h_d(T)$. Combining these two observations gives the lemma.
\end{proof}

We will need the following results on topological entropy.

\begin{proposition}[=Thm.\ 5 of \cite{AKM}: $\htop$ of quotients]
\label{prop:AKM5}
Let $X$ be a compact topological space and let $\pi: X\to \overline X$ be  a (continuous) quotient map (so defined by some equivalence relation on $X$).
If $\psi:X\to X$ is a continuous self-map which  descends to
$\overline{\psi}: \overline X\to  \overline X$, then $\htop( \overline \psi)\leq \htop(\psi)$.  
 \end{proposition}
 
 \begin{remark}[{\bf Anosov diffeomorphisms minimize entropy}]
 \label{remark:franks}
 Franks proved in \cite{Fra} the following: Let $f:T^n\to T^n$ be a diffeomorphism of an $n$-dimensional torus such that $f_*:H_1(T^n)\to H_1(T^n)$ has all eigenvalues outside the unit circle. Then $f$ is topologically semi-conjugate to a linear Anosov diffeomorphism 
 determined by the linear diffeomorphism $A(f):T^n\to T^n$ 
 induced by the map $f_*$.  Proposition \ref{prop:AKM5} implies that any $\htop(f)\geq \htop(A(f))$.  Thus Anosov diffeomorphisms of tori minimize entropy in their homotopy class.
 \end{remark}

\begin{proposition}[=Thm.\ 4 in \cite{AKM}: $\htop$ and invariant subspaces]
\label{prop:max}
Let $X$ be a topological space and suppose 
that $X=X_1\cup X_2$ for some closed subspaces $X_i\subset X$.  Let $\phi:X\to X$ be a continuous map such that $\phi(X_i)\subseteq X_i$ for each $i=1,2$.  Then
\[\htop(\phi)=\max\{\htop(\phi|_{X_1}),\htop(\phi|_{X_2})\}.\]
\end{proposition}

Bowen provides us in the situation of  Proposition \ref{prop:AKM5} with an inequality in the opposite direction:

\begin{proposition}[=Thm.\ 17 +Cor.\  18 of \cite{B}: $\htop$ and fibers]
\label{prop:B17}
In the situation of Proposition  \ref{prop:AKM5}, assume that $X$ and $\overline X$ are compact manifolds. Then 
\[
\htop(\psi)\leq \htop(\overline\psi)+\sup_{y\in \overline X}\htop(\psi,\pi^{-1}(y))\]
where $\htop(\psi,\pi^{-1}(y))$ denotes the topological entropy of $\psi$ relative to the compact set $\pi^{-1}(y)$ (see above).
In particular if $\overline\psi$ is the identity then  
\[
\htop(\psi)=\sup_{y\in Y}\htop(\psi_{\pi^{-1}(y)}).
\]
\end{proposition}

We will also need the following.  While surely known to experts, we could not find in the literature.  

\begin{proposition}
\label{proposition:rayaction}
Fix $n\geq 2$.  Let $A\in\GL_n(\Rb)$ and let 
$\psi_A:S^{n-1}\to S^{n-1}$ be the induced diffeomorphism on the space of rays through $0\in\Rb^n$.  Then $\htop(\psi_A)=0$.
\end{proposition}

When the eigenvalues of $A$ are all distinct, the diffeomorphism $\psi_A$ is Morse-Smale, and the result is well known. 

\begin{proof}
Consider the eigenspaces $V_i$ of $A$, and their 
corresponding sets $S(V_i)$ of rays, giving subspheres $S(V_i)\subseteq S^{n-1}$.   We note that for each $i$, the diffeomorphism $\psi_A$ acts on the sphere $S(V_i)$ by a rotation (which is $\pm$ the identity in case $V_i$ is the eigenspace for a real eigenvalue, the sign being that of the eigenvalue); in particular 
\begin{equation}
\label{eq:rays2}
\htop(\psi_A|V_i)=0.
\end{equation}  

Recall that the {\em nonwandering set} $\Omega(\psi_A)$ is defined to be the set of $x\in S^{n-1}$ such that given any neighborhood $U$ of $x$, there exists $n>0$ so 
that $f^n(U)\cap U\neq \emptyset$.  Bowen \cite{B2} showed that 
\begin{equation}
\label{eq:rays1}
\htop(S^{n-1})=\htop(\psi_A|\Omega(\psi_A)).
\end{equation}
The definition of eigenspace  
gives that 
\[\Omega(\psi_A)=\cup_iS(V_i).\]
This, together with \eqref{eq:rays1} and Proposition \ref{prop:max}, 
implies that 
\begin{equation}
\label{eq:rays3}
\htop(\psi_A)=\max\{\htop (\psi_A|S(V_i))\}.
\end{equation}

But \eqref{eq:rays2} gives that each term of the right-hand side of \eqref{eq:rays3} equals $0$, from which we conclude that 
$\htop(\psi_A)=0$, proving the proposition.
\end{proof}

\subsection{Computing $h_{\rm top}(f_T)$}
\label{subsection:htopmain}

In this subsection we prove Property 1 of Theorem \ref{theorem:main1}. We compute $\htop(f_T)$ by determining $\htop$ of the diffeomorphisms at each stage of the construction of $f_T$ given in \S\ref{subsection:construct} above. 

Corollary 2 of \cite{AKM}) states that if $\psi$ is any homeomorphism then $\htop(\psi^k)=|k|\htop(\psi)$ for all $k$.  Hence we can always take an appropriate power of any homeomorphism to assume that each member of a ``periodic set'' is invariant, compute the entropy of that, then divide. 
So for the purpose  of computing various entropies associated with $T$,  this  allows us to assume that $T$ leaves $A[2]$ pointwise fixed. Henceforth we make this assumption.

\para{Step 1($\htop(T_A)$ and $\htop(T_Y)$):}
Note that $T_A :A\to A$ is an endomorphism of the compact 
Lie group $A$ whose action on its Lie algebra  is just the action $T$ on $V$.  Corollary 16 of \cite{B} then tells us that 
\begin{equation}
\label{eq:evals3}
\htop(T_A)=\sum_{|\lambda_i|>1}\log |\lambda_i|.
\end{equation}
where the sum is over all eigenvalues $\lambda_i$ of $T$ of absolute value $>1$ (counted with multiplicity).  

Step 1 of the construction of $f_T$ given in \S\ref{subsection:construct} 
gives the following commutative diagram:
\begin{equation}
\label{eq:phihat}
\begin{array}{ccc}
Y&\stackrel{T_Y}{\longrightarrow}&Y\\
\pi \big\downarrow&&\big\downarrow\pi\\
A&\stackrel{T_A}{\longrightarrow}&A
\end{array}
\end{equation}

\begin{corollary}\label{cor:}
With the notation as above, $\htop(T_Y)=\htop(T_A)$.
\end{corollary}
\begin{proof}
Applying Proposition \ref{prop:AKM5} to 
the quotient map \eqref{eq:phihat} gives 
\begin{equation}
\label{eq:phihat2}
\htop(T_Y)\geq \htop(T_A).
\end{equation}
Proposition \ref{prop:B17} applied 
to \eqref{eq:phihat} gives that 
\begin{equation}
\label{eq:phihat3}
\htop(T_Y)\leq \htop(T_A)+\sup_{y\in A}\htop(T_Y,\pi^{-1}(y)).
\end{equation}
where the right-most term of \eqref{eq:phihat3} is the entropy of 
$T_Y$ relative to the compact set $\pi^{-1}(y)$ (cf. \S\ref{subsection:entropy}).  Since $T_A$ commutes with the translation group $A[2]$, the quantity $\sup_{y\in A}\htop(T_Y,\pi^{-1}(y))$ 
is realized over the origin $y=o\in A[2]$.   Note that $\pi^{-1}(o)$ 
is a $3$-sphere. Since 
$\pi^{-1} o$ is $T_Y$-invariant, Lemma \ref{lemma:relative1} 
implies that 
\[
\htop(T_Y,\pi^{-1}(o))=\htop(T_Y|_{\pi^{-1}(o)}).
\]
But $T_Y|_{\pi^{-1}(o)}$ is simply the induced action of the linear transformation  $T$ on the space of rays in $V$.   By Proposition \ref{proposition:rayaction} this diffeomorphism has 
topological entropy $0$.  Thus \eqref{eq:phihat3} becomes
\begin{equation}
\label{eq:phihat4}
\htop(T_Y)\leq \htop(T_A).
\end{equation}

Combining equations \eqref{eq:phihat2} and \eqref{eq:phihat4} gives 
$\htop(T_Y)=\htop(T_A)$, as desired.
\end{proof}

\para{Step 2 ($\htop(T_{Y'})$):}
By construction, $Y'$ is the union of $Y$ together with the 16 collars $[0,1]\times \partial_\alpha Y$.  Further, $T_{Y'}:Y'\to Y'$ leaves invariant the closed subset $Y$ and leaves invariant each closed subset $[0,1]\times \partial_\alpha Y$.  Proposition \ref{prop:max} thus implies that
\begin{equation}
\label{eq:gt}
\htop(T_{Y'})=\max\{\htop(T_Y),\htop(T_{Y'}|_{[0,1]\times \partial_0Y})\}.
\end{equation}
Again by construction, the action of $T_{Y'}$ on $[0,1]\times \partial_0Y\cong [0,1]\times S^3$ is given by 
$T_{Y'}(t,z):=(t, \gamma_T(t)(z))$ where $\gamma_T$ is a smooth path in $\SL(V)$ from $T$ to ${\rm Id}$.   

Apply the last clause  of Proposition \ref{prop:B17} to the map $T_{Y'}:[0,1]\times \partial_0Y\rightarrow [0,1]\times \partial_0Y$ and the projection onto $[0,1]$.  
Since the linear action of each $\gamma_T(t)$ on the space of rays $\partial_0Y\cong S^3$ has entropy $0$ (by Proposition \ref{proposition:rayaction}), it follows that 
\begin{equation}
\label{eq:gt2}
\htop(T_{Y'}|_{[0,1]\times \partial_0Y})=0.
\end{equation}
Since $T_{Y'}|_Y=T_Y$ by definition, Equations \eqref{eq:gt} and \eqref{eq:gt2} imply that
\begin{equation}
\label{eq:gt3}
\htop(T_{Y'})=\max\{\htop(T_{Y'}|Y),\htop(T_{Y'}|_{\partial_\beta Y\times [0,1]})\}=\max\{\htop(T_Y),0\}=\htop(T_Y).
\end{equation}

Now consider the commutative diagram coming from Step 4, the Hopf collapse:
\begin{equation}
\label{eq:Hc1}
\begin{array}{ccc}
Y'&\stackrel{T_{Y'}}{\longrightarrow}&Y'\\
\pi \big\downarrow&&\big\downarrow\pi\\
\overline Y'&\stackrel{T_{\overline{Y}'}} {\longrightarrow}&\overline{Y}'
\end{array}
\end{equation}
where $\pi$ is the map that Hopf-collapses each $\partial_\alpha Y_\alpha\times\{1\}$.   Then 
\[
\htop(T_{\overline Y'}|_{\overline{[0,1]\times \partial_0Y}})\leq \htop(T_{Y'}|_{[0,1]\times \partial_0Y]})=0
\]
where the first inequality is Proposition \ref{prop:AKM5} and the second equality is Equation \ref{eq:gt2}. 

\para{Step 3 ($\htop(T_{\overline{Y}'}))$.}
Now we can compute the entropy of $T_{\overline Y'}$ exactly as we did $T_{Y'}$, with the one change being that must compute 
$\htop(T_{\overline Y'}|_{\overline{[0,1]\times \partial_0Y}})$ instead of 
$\htop(T_Y|_{[0,1]\times \partial_0Y})$.  Proposition \ref{prop:max} then gives
\[
\htop(T_{\overline Y'})=
\max\{\htop(T_Y),\htop(T_{\overline Y'}|_{\overline{[0,1]\times \partial_0Y}})\}
=\max\{ \htop(T_Y),0\}=\htop(T_Y).
\]

\para{Step 4 ($\htop(f_T)$).} The involution $(-1)_{\overline Y'}$ of $\overline Y'$ commutes with 
$T_{\overline Y'}$, giving the following commutative diagram:
\begin{equation}
\begin{array}{ccc}
\overline Y'&\stackrel{T_{\overline Y'}}{\longrightarrow}&\overline Y'\\
\pi \big\downarrow&&\big\downarrow\pi\\
M&\stackrel{f_T}{\longrightarrow}&M
\end{array}
\end{equation}

Proposition \ref{prop:AKM5} gives 
\begin{equation}
\label{eq:ft1}
\htop(f_T)\leq \htop(T_{\overline Y'})
\end{equation}

Proposition \ref{prop:B17} gives 
\[\htop(T_{\overline Y'})\leq \htop(f_T)+\sup_{x\in M}\htop(T_{\overline Y'},\pi^{-1}(x)).\]
Since $\pi^{-1}(x)$ is $T_{\overline Y'}$-invariant and since  $\pi^{-1}(x)$ is a finite set:
\[
\htop(T_{\overline Y'},\pi^{-1}(x))=\htop(T_{\overline Y'}|_{\pi^{-1}(x)})=0
\]
so that
\begin{equation}
\label{eq:ft2}
\htop(T_{\overline Y'})\leq \htop(f_T)+0=\htop(f_T).
\end{equation}

Combining \eqref{eq:ft1} and \eqref{eq:ft2} gives
\[
\htop(f_T)=\htop(T_{\overline Y'}).
\]


Combining all of the steps above gives
\[
\textstyle \htop(f_T)=\htop(T_{\overline Y'})=\htop(T_{Y'})=\htop(T_Y)=\sum_{|\lambda_i|>1}\log |\lambda_i|.
\]


\begin{thebibliography}{ABCDEF}
\small

\bibitem[AKM]{AKM}
R.\ L.\ Adler, A.\ G.\ Konheim, M.\ H.\ McAndrew:
\textit{Topological Entropy,} 
Trans.\ Amer.\ Math.\ Soc. \textbf{114}, 1965, 309--319.

\bibitem[B]{B}
R. Bowen: 
\textit{Entropy for Group Endomorphisms and Homogeneous Spaces,}
Trans.\ Amer.\ Math.\ Soc. \textbf{153} 1971,  401--414.

\bibitem[B2]{B2}
R. Bowen: 
\textit{Topological entropy and Axiom A,}
in \textsl{Global Analysis,} Proc.\ Sympos.\ Pure Math., Amer.\ Math.\ Soc.\ \textbf{14}, 23--42 (1970)

\bibitem[Ca]{Ca}
S. Cantat,
\textit{Dynamique des automorphismes des surfaces K3}, 
Acta Math. 187 (2001), no. 1, 1--57.


\bibitem[FL1]{FL1}
B.~Farb,  E.~Looijenga,
\textit{
The Nielsen Realization problem for K3 surfaces}, 
J.\ Differential  Geom.\ \textbf{127}, 505--549 (2024).

\bibitem[FL2]{FL2}
B.~Farb, E.~Looijenga,
\textit{The smooth Mordell-Weil group and mapping class groups of elliptic surfaces}, 
preprint, May 2024, \url{https://arxiv.org/pdf/2403.15960}.

\bibitem[FLP]{FLP}
A. Fathi, F. Laudenbach, and V. P\'{o}enaru, editors,
\textit{Travaux de Thurston sur les surfaces}, 
Ast\'{e}risque, Vol. 66, Soc. Math. de France, 1979.

\bibitem[Fr]{Fr}
D. Fried,
\textit{Word maps, Isotopy and Entropy},
Trans.\ Amer.\ Math.\ Soc.\ \textbf{296}. 851--859 (1986).

\bibitem[Fra]{Fra}
Franks, John
\textit{Anosov diffeomorphisms}, in Global Analysis (Proc. Sympos. Pure Math.), Vol. XIV,  1968, pp. 61--93.

\bibitem[Gr]{Gr}
M.~Gromov,
\textit{Entropy, homology and semialgebraic geometry}, S\'{e}minaire Bourbaki exp.\ 663. 
Ast\'{e}risque \textbf{145-146}, 225--240 (1987). 

\bibitem[LP]{LP}
E.~Looijenga, C.~ Peters:
\textit{Torelli theorems for K\"ahler K3 surfaces,}
Compos.\  Math.\  \textbf{42} (1980/81),  145--186.

\bibitem[Mc]{Mc}
C.~McMullen: 
\textit{Dynamics on K3 surfaces: Salem numbers and Siegel disks,}
J.\  Reine Angew.\ Math.\ \textbf{545} (2002), 201--233.

\bibitem[Mc2]{Mc2}
C.~McMullen: 
\textit{Dynamics on blowups of the projective
plane,} 
Publ.\ Math.\ Inst.\ Hautes \'{E}tudes Sci.\ \textbf{105} (2007), 49--89.

\bibitem[Sh]{Sh}
M.~Shub:
\textit{All, most, some differentiable dynamical systems}, 
Proc. ICM, Vol. III, EMS, Z\"urich, 2006, 99--120.

\bibitem[Si]{Si}
Y.-T.~Siu:
\textit{Every K3 surface is K\"ahler,} 
Invent.\ Math.\ \textbf{73} (1983), 139--150. 

\bibitem[SS]{SS}
M.~Shub,  D.~Sullivan,
\textit{Homology theory and dynamical systems}, 
Topology \textbf{14} (1975), 109--132. 

\bibitem[Y]{Y}
Y.~Yomdin: 
\textit{Volume growth and entropy,}
Israel J.\ Math.\ \textbf{57} (1987), 285--300.

\end{thebibliography}
\end{document}